\documentclass[12pt,reqno]{amsart}
\usepackage{amssymb,amsmath,amsthm,fullpage,enumerate}
\usepackage{xcolor}
\usepackage{biblatex} 
\usepackage{comment}
\usepackage{cases}

\usepackage[OT2,OT1]{fontenc}

\DeclareFontFamily{OT2}{cmr}{\hyphenchar\font45 }
\DeclareFontShape{OT2}{cmr}{m}{n}{<->wncyr10}{}
\DeclareFontShape{OT2}{cmr}{m}{it}{<->wncyi10}{}
\DeclareFontShape{OT2}{cmr}{m}{sc}{<->wncysc10}{}
\DeclareFontShape{OT2}{cmr}{b}{n}{<->wncyb10}{}
\DeclareFontShape{OT2}{cmr}{bx}{n}{<->ssub*wncyr/b/n}{}

\DeclareFontFamily{OT2}{cmss}{\hyphenchar\font45 }
\DeclareFontShape{OT2}{cmss}{m}{n}{<->wncyss10}{}

\DeclareRobustCommand\cyr{\fontencoding{OT2}\selectfont}
\DeclareTextFontCommand{\textcyr}{\cyr}

\DeclareUnicodeCharacter{0306}{}

\newtheorem{theorem}{Theorem}[section]

\newtheorem{corr}[theorem]{Corollary}

\theoremstyle{definition}

\newtheorem{deff}[theorem]{Definition}

\theoremstyle{remark}
\newtheorem{comm}[theorem]{Remark}

\newcommand{\R}{\mathbb R}

\newcommand{\C}{\mathbb C}

\newcommand{\p}{\partial}

\newcommand{\z}{\bar z}

\DeclareMathOperator{\im}{Im}
\DeclareMathOperator{\re}{Re}

\newcommand{\Hpn}{H^{n,p}(D)}

\newcommand{\Mpn}{H^{n,p}_A(D)}

\title{Hardy Spaces of Meta-Analytic Functions and the Schwarz Boundary Value Problem}

\author{William L. Blair}
\address{Department of Mathematical Sciences\\
  University of Arkansas\\
  Fayetteville, Arkansas}
\email{wlblair@uark.edu}



\begin{document}

\begin{abstract}
    We extend representation formulas that generalize the similarity principle of solutions to the Vekua equation to certain classes of meta-analytic functions. Also, we solve a generalization of the higher-order Schwarz boundary value problem in the context of meta-analytic functions with boundary conditions that are boundary values in the sense of distributions. 
\end{abstract}

\maketitle

\section{Introduction}
    In this paper, we prove a representation formula for certain classes of meta-analytic functions and solve an associated Schwarz boundary value problem.

    We work to further the study of solutions to generalizations of the Cauchy-Riemann equation. One of the most well studied generalization is the Vekua equation
    \begin{equation}
        \frac{\p w}{\p\z} = Aw + B\overline{w} \label{vekeqn},
    \end{equation}
    for $A,B \in L^q$, $q>2$, see \cite{Vek}. Solutions of (\ref{vekeqn}) are called generalized analytic functions and share many of the desirable characteristics of complex analytic functions because of the representation known as the similarity principle. The similarity principle is the representation of a generalized analytic function as a factorization 
    \[
        w = e^\varphi \phi,
    \]
    where $\phi$ is holomorphic and $\varphi$ is H\"older continuous on the closure of the domain. Since H\"older continuous functions on a closed set are bounded in modulus, the similarity principle not only extends properties of generic holomorphic functions that depend on size but extends those kinds of properties of the Hardy spaces of holomorphic functions when the holomorphic factor is an element of one of these spaces, see \cite{KlimBook}. The poly-analytic (or $n$-analytic) functions are those functions that solve the higher-order generalization of the Cauchy-Riemann equation
    \begin{equation*}
        \frac{\p^n f}{\p\z^n} = 0.
    \end{equation*}
    These functions are known to be representable as a polynomial in $\z$ with holomorphic coefficients. When the holomorphic coefficients are Hardy space functions, the resulting classes of functions inherit some of the properties of the corresponding Hardy space, see \cite{polyhardy}. Considering the Vekua equation as 
    \begin{equation*}
        \left(\frac{\p }{\p\z} -A -BC(\cdot)\right) w = 0,
    \end{equation*}
    where $C(\cdot)$ denotes the mapping that sends functions to their complex conjugate, it is natural to consider the higher-order generalizations
    \begin{equation}
        \left(\frac{\p }{\p\z} -A -BC(\cdot)\right)^n w = 0, \label{highervek}
    \end{equation}
    for $n >1$. In \cite{metahardy}, the authors show that solutions to (\ref{highervek}) with $A \in \C$ and $B\equiv 0$ are representable as 
    \[
        w = \sum_{k=0}^{n-1} \z^k e^\varphi \phi_k,
    \]
    i.e., a polynomial in $\z$ with coefficients which are solutions to the Vekua equation (\ref{vekeqn}). In \cite{metahardy}, the authors call these functions meta-analytic and show that when the generalized analytic function coefficients are members of the generalized Hardy spaces from \cite{KlimBook}, then these classes of functions inherit properties of the Hardy spaces. We extend this representation to the solutions of (\ref{highervek}) with $A$ a member of the $W^{n-1,q}$ Sobolev space, $q>2$, and $B\equiv 0$, and we show that, with $A \in W^{n-1,\infty}$, the extension of Hardy space boundary behavior is present in this case too. 

    Next, we consider the Schwarz boundary value problem. The Schwarz boundary value problem is a classically studied simplification of the Riemann-Hilbert problem in the complex plane, see \cite{BegBook}. The author, in \cite{WB} and \cite{WB2}, extended the solvable classes of the Schwarz boundary value problem to those with boundary conditions in terms of boundary values in the sense of distributions for solutions of nonhomogeneous Cauchy-Riemann equations
    \begin{equation}
        \frac{\p w}{\p\z} = f, \label{nonhcr}
    \end{equation}
    and the higher order generalizations 
    \begin{equation}
        \frac{\p^n w}{\p\z^n} = f, \label{nonhcrhigher}
    \end{equation}
    with $f$ an integrable function. In \cite{WB2}, a special case was considered that employed the structure of $n$-analytic functions to solve a  Schwarz boundary value problem where the $f$ in (\ref{nonhcr}) and (\ref{nonhcrhigher}) is not necessarily integrable, the first result of its kind in the literature. We utilize the construction from the results in \cite{WB2} and the representations we prove for meta-analytic functions to solve a Schwarz boundary value problem where the solution is meta-analytic and all of the boundary conditions are in terms of boundary values in the sense of distributions. The standard technique of solving the Schwarz boundary value problem is to use singular integral operators, see \cite{Vek},\cite{BegBook}, or \cite{ Beg}. The novelty of our technique is to avoid the use of these integral operators in the places where their use would require greater boundary regularity. We are not restricted to continuous boundaries and can consider the more general case where the boundary condition is in terms of only boundary values in the sense of distributions.This work extends the results of \cite{metahardy}, the classical case of continuous boundary condition, and the case of $A\equiv 0$ from \cite{WB2}.

    The paper is structured as follows. Section 2 provides definitions of the classes of functions that we will encounter throughout the paper and background results. In Section 3, we prove the generalization of the similarity principle for solutions of $\left( \frac{\p}{\p\z} - A\right)^n f = 0$ with $A$ nonconstant. Also, we show the improvement that occurs to the representation when we consider the solution in the context of Hardy spaces and that certain desirable boundary behaviors are recovered. In Section 4, we consider a generalization of the Schwarz boundary value problem for meta-analytic functions and with boundary conditions in terms of only boundary values in the sense of distributions of holomorphic functions and solve the problem explicitly. Certain connections between the solutions of the boundary value problem and the Hardy spaces of meta-analytic functions are described. 

    We thank Professor Andrew Raich and Professor Gustavo Hoepfner for their support during the time this work was produced. 

\section{Definitions and background}

We represent the unit disk in the complex plane by $D$, and its boundary by $\p D$. We represent the  Sobolev spaces of $L^q(D)$ functions with $k$ weak derivatives which are all in $L^q(D)$ by $W^{k,q}(D)$. We represent the space of distributions on $\p D$ by $\mathcal{D}'(\p D)$. By $C^{0,\alpha}(S)$, we denote the set of $\alpha$-H\"older continuous functions defined on the set $S$. We define the classes of functions that are used throughout.

\begin{deff}
We denote by $H(D)$ the set of holomorphic functions on $D$, i.e., $f: D \to \C$ such that
\[
\frac{\p f}{\p\z} = 0.
\]
\end{deff}

\begin{deff}\label{bvcircle}Let $f$ be a function defined on $D$. We say that $f$ has a boundary value in the sense of distributions, denoted by $f_b \in \mathcal{D}'(\p D)$, if, for every $ \varphi \in C^\infty(\partial D)$, the limit
            \[
             \langle f_b, \varphi \rangle := \lim_{r \nearrow 1} \int_0^{2\pi} f(re^{i\theta}) \, \varphi(\theta) \,d\theta
            \]
            exists.
            
\end{deff}

\begin{deff}
We define $H_b$ to be that subset of functions in $H(D)$ that have boundary values in the sense of distributions.
\end{deff}

The next theorem gives a growth condition which guarantees a holomorphic function has a boundary value in the sense of distributions and provides a representation formula which we use in Section \ref{sbvp}. 

\begin{theorem}[Theorem 3.1 \cite{GHJH2}]\label{GHJH23point1}
For $f \in H(D)$, the following are equivalent:
\begin{enumerate}
    \item For every $\phi \in C^\infty(\p D)$, there exists the limit
    \[
             \langle f_b, \phi \rangle := \lim_{r \nearrow 1} \int_0^{2\pi} f(re^{i\theta}) \, \phi(\theta) \,d\theta.
            \]

    \item There is a distribution $f_b \in \mathcal{D}'(\p D)$ such that $f$ is the Poisson integral of $f_b$
    \[
             f(re^{i\theta}) = \frac{1}{2\pi}\langle f_b, P_r(\theta - \cdot) \rangle,
            \]
    where 
    \[
        P_r(\theta) = \frac{1-r^2}{1-2r\cos(\theta) +r^2}
    \]
    is the Poisson kernel on $D$. 

    \item There are constants $C>0$, $\alpha \geq 0$, such that 
    \[
        |f(re^{i\theta})| \leq \frac{C}{(1-r)^\alpha},
    \]
    for $0 \leq r < 1$.
\end{enumerate}
\end{theorem}

\begin{deff}
For $0 < p < \infty$, we denote by $H^p(D)$ the holomorphic Hardy spaces of functions $f \in H(D)$ such that 
\[
||w||_{H^p(D)}:= \left(\sup_{0< r < 1} \int_0^{2\pi} |w(re^{i\theta})|^p \,d\theta\right)^{1/p} < \infty.
\]
\end{deff}

\begin{theorem}[Corollary 3.1 \cite{GHJH2}]
The functions in $H^p(D)$, $0 < p \leq \infty$, satisfy (3) in Theorem \ref{GHJH23point1}
\end{theorem}

So the functions in $H^p(D)$ have boundary values in the sense of distributions, but it is pointed out in \cite{GHJH2} that $\cup_{0< p \leq \infty} H^p(D)$ is a proper subset of $H_b$ (see \cite{Duren} for an example of $h \in H_b$ and $h\not\in \cup_{0< p \leq \infty} H^p(D)$).

A classical result about the boundary behavior of functions in $H^p(D)$ is the following.

\begin{theorem}[\cite{Duren}]\label{bvcon}
A function $w(z) \in H^p(D)$, $0 < p < \infty$, has nontangential boundary values $w_+(e^{i\theta}) \in L^p(\partial D)$ at almost all points $e^{i\theta}$ of the circle $\partial D$, 
\[
\lim_{r\nearrow 1} \int_0^{2\pi} |w(re^{i\theta})|^p \, d\theta = \int_0^{2\pi} |w_+(e^{i\theta})|^p \,d\theta,
\]
and
\[
\lim_{r\nearrow 1} \int_0^{2\pi} |w(re^{i\theta})- w_+(e^{i\theta})|^p \, d\theta = 0.
\]
\end{theorem}

\begin{deff}
For $n$ a positive integer and $0 < p < \infty$, we define the poly-Hardy space $\Hpn$ to be the set of functions $f$ that satisfy
\[
\frac{\p^n f}{\p\z^n} = 0.
\]
and 
\[
||f||_{n,p} := \sum_{k=0}^{n-1} \left( \sup_{0 < r < 1} \int_0^{2\pi} \left| \frac{\p^k f}{\p\z^k}(re^{i\theta}) \right|^p d\theta \right)^{1/p} < \infty.
\]
\end{deff}

\begin{deff}
For $n$ a positive integer, $A \in W^{n-1,q}(D)$, $q>2$, and $0 < p < \infty$, we define the meta-Hardy space $\Mpn$ to be the set of functions $f: D \to \mathbb{C}$ that satisfy
\[
\left(\frac{\p }{\p\z}- A\right)^n  f= 0,
\]
and
\[
||f||_{n,p} := \sum_{k=0}^{n-1} \left( \sup_{0 < r < 1} \int_0^{2\pi} \left| \frac{\p^k f}{\p\z^k}(re^{i\theta}) \right|^p d\theta \right)^{1/p} < \infty.
\]
\end{deff}

Clearly, $H^{n,p}_0(D) = H^{n,p}(D)$.

We recall a well-known fact about nonhomogeneous Cauchy-Riemann equations.

\begin{theorem}[\cite{Vek},\cite{BegBook}]\label{toperator}
For any $f \in L^1(D)$, 
\[
g(z) = -\frac{1}{\pi} \iint_{D} \frac{f(\zeta)}{\zeta - z}\,d\xi\,d\eta,
\]
where $\zeta = \xi + i \eta$, solves
\[
\frac{\p g}{\p\z} = f, 
\]
and for $f \in L^q(D)$, $q>2$, $g \in C^{0,\alpha}(\overline{D})$, $\alpha = \frac{q-2}{q}$.
\end{theorem}

There exist alternative integral representations for solutions to nonhomogeneous Cauchy-Riemann equations that satisfy other conditions. One that is useful when solving Schwarz boundary value problems, see Section \ref{sbvp}, is the following.

\begin{theorem}[\cite{Beg}]\label{toperatoralt}
For any $f \in L^1(D)$, 
\[
g(z) = -\frac{1}{\pi} \iint_{D} \left( \frac{f(\zeta)}{\zeta}\frac{\zeta + z}{\zeta - z} + \frac{\overline{f(\zeta)}}{\overline{\zeta} }\frac{1+z\overline{\zeta}}{1-z\overline{\zeta}} \right)\,d\xi\,d\eta,
\]
where $\zeta = \xi + i \eta$, solves
\[
\frac{\p g}{\p\z} = f, 
\]
and $\im{g(0)} = 0$.
\end{theorem}

\begin{comm}
Note that the two integral representations above differ by a holomorphic function. We will later appeal to results from \cite{Vek} that improve the regularity of the first integral representation from Theorem \ref{toperator}, and these results will extend to the integral representation from Theorem \ref{toperatoralt} by this fact. 
\end{comm}

\section{Representation Theorems}

The next theorem is a classic result that has been widely communicated by M. Balk.

\begin{theorem}[\cite{Balk}]
Every function $f$ that satisfies 
\[
\frac{\p^n f}{\p\z^n} = 0
\] is representable by 
\[
f(z) = \sum_{k=0}^{n-1} \z^k f_k(z),
\]
with $f_k \in H(D)$.
\end{theorem}

The classic result is improved in \cite{polyhardy} for functions in the spaces $\Hpn$.

\begin{theorem}[Theorem 2.1 \cite{polyhardy}]\label{polyhardyrep}
For $n$ a positive integer and $0 < p < \infty$, every $f \in \Hpn$ is representable as 
\[
f(z) = \sum_{k=0}^{n-1} \z^k f_k(z),
\]
with $f_k \in H^p(D)$. 
\end{theorem}

The next two theorems extend Theorem 2.1 from \cite{metahardy} where $A$ is a complex constant, with a slightly different form. We prove the first using the argument of Balk from \cite{Balk} for the $A\equiv 0$ case. We prove the second using the argument from the proof of Theorem 2.1 in \cite{metahardy}.

\begin{theorem}\label{metarep}
For $n$ a positive integer and $A \in W^{n-1,q}(D)$, $q>2$, every solution $w$ of 
\begin{equation}
\left( \frac{\p}{\p\z} - A\right)^n  w = 0 \label{metapde}
\end{equation}
has the form 
\begin{equation}
w(z) = e^{\psi(z)}\sum_{k=0}^{n-1} \z^k w_k,
\end{equation}
where 
\[
\psi(z) = -\frac{1}{\pi} \iint_D \frac{A(\zeta)}{\zeta - z}\,d\xi\,d\eta
\]
and $w_k \in H(D)$, for every $k$.
\end{theorem}

\begin{proof}
The case $n = 1$ is a classic and can be found in \cite{Vek}.

We proceed by induction. Suppose that the theorem holds for all $n$ such that $1\leq n \leq m-1$. Let $w$ be a solution to 
\begin{equation*}
\left( \frac{\p}{\p\z} - A\right)^m  w = 0.
\end{equation*}
So, $f = \left( \frac{\p}{\p\z} - A\right)  w $ solves $\left( \frac{\p}{\p\z} - A\right)^{m-1}  f = 0$, and 
\[
f = \left( \frac{\p}{\p\z} - A\right)w = e^{\psi(z)}\sum_{k=0}^{m-2} \z^k w_k
\]
with $w_k \in H(D)$, $0 \leq k \leq m-2$. Observe that 
\[
g = e^{\psi(z)}\sum_{k=0}^{m-2} \frac{1}{k+1}\z^{k+1} w_k
\]
also solves  $\left( \frac{\p}{\p\z} - A\right)^{m}  g = 0$ by direct computation. Consider
\begin{align*}
&\left( \frac{\p}{\p\z} - A\right)\left(w - e^{\psi(z)}\sum_{k=0}^{m-2} \frac{1}{k+1}\z^{k+1} w_k\right) \\
&= e^{\psi(z)}\sum_{k=0}^{m-2} \z^k w_k - \left( \frac{\p}{\p\z} - A\right)\left(e^{\psi(z)}\sum_{k=0}^{m-2} \frac{1}{k+1}\z^{k+1} w_k\right) \\
&= e^{\psi(z)}\sum_{k=0}^{m-2} \z^k w_k - \left(A(z)e^{\psi(z)}\sum_{k=0}^{m-2} \frac{1}{k+1}\z^{k+1} w_k + e^{\psi(z)}\sum_{k=0}^{m-2} \z^k w_k -A(z)e^{\psi(z)}\sum_{k=0}^{m-2} \frac{1}{k+1}\z^{k+1} w_k \right) \\ & = 0.
\end{align*}
Hence, the difference is a solution to the $n = 1$ case, and 
\[
w - e^{\psi}\sum_{k=0}^{m-2} \frac{1}{k+1}\z^{k+1} w_k  = e^{\psi} \phi_o,
\]
where $\phi_0 \in H(D)$. Let $\phi_k = \frac{1}{k+1}w_k$ for $0\leq k \leq m-2$, and we have 
\[
w = e^{\psi}\sum_{\ell=0}^{m-1}\z^\ell \phi_\ell,
\]
where $\phi_\ell\in H(D)$, for $0\leq \ell\leq m-1$. 
\end{proof}

\begin{theorem}\label{metahardyrep}
For $n$ a positive integer, $0 < p < \infty$, and $A \in W^{n-1,\infty}(D)$, every function $w \in H^{n,p}_{A}(D)$ has the form 
\begin{equation}
w(z) = e^{\psi(z)}\sum_{k=0}^{n-1} \z^k w_k(z),
\end{equation}
where 
\[
\psi(z) = -\frac{1}{\pi} \iint_D \frac{A(\zeta)}{\zeta - z}\,d\xi\,d\eta
\]
and $w_k \in H^p(D)$, for every $k$.
\end{theorem}

\begin{proof}
Let $f \in \Mpn$. So, by Theorem \ref{metarep}, $f(z) = e^{\psi}F(z)$, where $F(z)  = \sum_{k = 0}^{n-1}\z^k f_k(z)$, $f_k \in H(D)$, and \[
\psi(z) = -\frac{1}{\pi} \iint_D \frac{A(\zeta)}{\zeta - z}\,d\xi\,d\eta.
\]  So, 
\begin{align*}
\begin{pmatrix}
f(z) \\
\frac{\p f}{\p \z}(z) \\
\vdots \\
\frac{\p^{n-1} f}{\p\z^{n-1}}(z) 
\end{pmatrix}
    &= 
\begin{pmatrix}
e^{\psi(z)}F(z) \\
\frac{\p }{\p \z} \left( e^{\psi(z)}F(z) \right)  \\
\vdots \\
\frac{\p^{n-1} }{\p\z^{n-1}} \left( e^{\psi(z)}F(z)  \right)
\end{pmatrix} \\
&= e^{\psi(z)}\begin{pmatrix} 
        1 & 0 & 0 & \cdots & 0\\
       A & 1 & 0 & \cdots & 0 \\
       A^2 + \frac{\p A}{\p\z} & 2A & 1 & \cdots & 0 \\
        \vdots  & \cdots & \ddots & \vdots\\
        P_{(n-1,1)}(A) & P_{(n-1,2)}(A) &  P_{(n-1,3)}(A) & \cdots & 1 
    \end{pmatrix}
    \begin{pmatrix}
        F(z) \\
        \frac{\p F}{\p\z}(z) \\
        \frac{\p^2 F}{\p\z^2}(z) \\
        \vdots \\
        \frac{\p^{n-1} F}{\p\z^{n-1}}(z)
    \end{pmatrix},
\end{align*}
where we note that the matrix is lower triangular, each $P_{(k,j)}(A)$, where $(k,j)$ denotes that the polynomial is the $(k,j)$ entry of the matrix, is a polynomial of order at most $\max\{k-1,j-1\}$ in $A$ and its derivatives up to order at most $\max\{k-1, j - 1\}$, and the entries on the diagonal are all 1. Hence, the matrix, which we will call $[A]$, is invertible, and 
\begin{align*}
 \begin{pmatrix}
        F(z) \\
        \frac{\p F}{\p\z}(z) \\
        \frac{\p^2 F}{\p\z^2}(z) \\
        \vdots \\
        \frac{\p^{n-1} F}{\p\z^{n-1}}(z)
    \end{pmatrix}
    &= 
    e^{-\psi(z)} [A]^{-1} 
    \begin{pmatrix}
f(z) \\
\frac{\p f}{\p \z}(z) \\
\vdots \\
\frac{\p^{n-1} f}{\p\z^{n-1}}(z) 
\end{pmatrix}.
\end{align*}
Since $\psi$, $A$, and all derivatives of $A$ up to order $n -1$ are bounded, it follows that 
\begin{align*}
\left|\left| \frac{\p^j F}{\p\z^j} \right|\right|^p_{H^p(D)}
&= \sup_{0 < r < 1} \int_0^{2\pi} \left| \frac{\p^j F}{\p\z^j}(re^{i\theta}) \right|^p \, d\theta \\
&= \sup_{0 < r < 1} \int_0^{2\pi} \left|e^{-\psi(re^{i\theta})} \sum_{k = 0}^{n-1} [A]^{-1}_{j,k}(re^{i\theta}) \frac{\p^k f}{\p\z^k}(re^{i\theta}) \right|^p \, d\theta \\
&\leq C \sum_{k = 0}^{n-1} \sup_{0 < r < 1} \int_0^{2\pi} \left| \frac{\p^k f}{\p\z^k}(re^{i\theta}) \right|^p \, d\theta < \infty, 
\end{align*}
for each $k$. Hence, $F \in \Hpn$. By Theorem \ref{polyhardyrep}, $F$ can be represented as 
\[
F(z) = \sum_{k = 0}^{n-1} \z^k f_k(z),
\]
where each $f_k \in H^p(D)$. Thus, 
\[
f(z) = e^{\psi(z)}\sum_{k = 0}^{n-1} \z^k f_k(z),
\]
where each $f_k \in H^p(D)$. 
\end{proof}

While the results in Section \ref{sbvp} are in terms of boundary values in the sense of distributions, we show that the representation from Theorem \ref{metahardyrep} allows us to prove that the functions in $\Mpn$ have nontangential boundary values almost everywhere and the functions converge to those boundary values in the $L^p$ norm.  

\begin{theorem}\label{metabv}
For $n$ a positive integer, $0< p < \infty$, and $A \in W^{n-1,\infty}(D)$, every $f \in \Mpn$ has a nontangential boundary value $f_+ \in L^p(\p D)$ almost everywhere on $\p D$ and
\[
\lim_{r\nearrow 1} \int_0^{2\pi} \left| f_+(e^{i\theta})  - f(re^{i\theta}) \right|^p \,d\theta = 0.
\]
\end{theorem}

\begin{proof}
To show that the nontangential boundary value exists and the function converges to that nontangential boundary value in the $L^p(\p D)$ norm, we follow the argument of Theorem 2.3 from \cite{metahardy}.

Let $f \in \Mpn$ have the representation $f(z) = e^{\psi(z)}\sum_{k=0}^{n-1} \z^k f_k(z)$ from Theorem \ref{metahardyrep}. By Theorem \ref{bvcon}, each $f_k \in H^p(D)$ has a nontangential boundary value $f_{k+}$ almost everywhere on $\p D$, $f_{k+} \in L^p(\p D)$, and 
\begin{equation}\label{hardyfunctionsconvonbound}
\lim_{r \nearrow 1} \int_0^{2\pi} |f_{k+}(e^{i\theta}) - f_k(re^{i\theta})| \, d\theta  = 0.
\end{equation}

Since $e^{\psi}$ and $\z^k$ are continuous up to the boundary, it follows that 
\[
f_+(e^{i\theta}) = e^{\psi(e^{i\theta})}\sum_{k=0}^{n-1} e^{-ik\theta} f_{k+}(e^{i\theta}), 
\]
and 
\begin{align*}
\int_{0}^{2\pi} |f_+(e^{i\theta})|^p \,d\theta 
&= \int_{0}^{2\pi} |e^{\psi(e^{i\theta})}\sum_{k=0}^{n-1} e^{-ik\theta} f_{k+}(e^{i\theta})|^p \,d\theta \leq 
 C \sum_{k=0}^{n-1} \int_{0}^{2\pi} |f_{k+}(\theta)|^p \,d\theta < \infty,
\end{align*}
where $C$ is a constant that only depends on $p$, $n$, and $A$. 

Now, observe 
\begin{align*}
&\int_0^{2\pi} |f_+(e^{i\theta}) - f(re^{i\theta})|^p \,d\theta \\
&= \int_0^{2\pi} \left|e^{\psi(e^{i\theta})}\sum_{k=0}^{n-1} e^{-ik\theta} f_{k+}(e^{i\theta}) - e^{\psi(re^{i\theta})}\sum_{k=0}^{n-1} r^k e^{-ik\theta} f_k(re^{i\theta}) \right|^p \,d\theta \\
&\leq C\int_0^{2\pi} \left|e^{\psi(e^{i\theta})}\sum_{k=0}^{n-1} e^{-ik\theta} f_{k+}(e^{i\theta}) - e^{\psi(e^{i\theta})}\sum_{k=0}^{n-1} e^{-ik\theta} f_k(re^{i\theta}) \right|^p \,d\theta \\
&\quad \quad + C\int_0^{2\pi} \left|e^{\psi(e^{i\theta})}\sum_{k=0}^{n-1} e^{-ik\theta} f_k(re^{i\theta}) - e^{\psi(re^{i\theta})}\sum_{k=0}^{n-1} r^k e^{-ik\theta} f_k(re^{i\theta}) \right|^p \,d\theta \\
&= C\int_0^{2\pi} \left|e^{\psi(e^{i\theta})} \sum_{k=0}^{n-1}  \left( f_{k+}(e^{i\theta})-  f_k(re^{i\theta}) \right)\right|^p \,d\theta  + C\int_0^{2\pi} \left|\sum_{k=0}^{n-1} \left( e^{\psi(e^{i\theta})} - e^{\psi(re^{i\theta})} r^k\right) f_k(re^{i\theta}) \right|^p \,d\theta \\
&\leq C \sum_{k=0}^{n-1} \int_0^{2\pi} \left|   f_{k+}(e^{i\theta})-  f_k(re^{i\theta}) \right|^p \,d\theta + C_k\sum_{k=0}^{n-1} \int_0^{2\pi} |e^{\psi(e^{i\theta})} - e^{\psi(re^{i\theta})}|^p \, |f_k(re^{i\theta})|^p \,d\theta.
\end{align*}
Note that the left sum in the right hand side approaches zero as $r \nearrow 1$, because of (\ref{hardyfunctionsconvonbound}), and the right sum in the right hand side approaches zero as $r \nearrow 1$ by the Dominated Convergence Theorem. So, 
\begin{equation*}\label{metahardyfunctionsconvonbound}
\lim_{r\nearrow 1} \int_0^{2\pi} \left| f_+(e^{i\theta})  - f(re^{i\theta}) \right|^p \,d\theta = 0.
\end{equation*}

\end{proof}

Next, we show that certain classes of solutions to equation (\ref{metapde}) have boundary values in the sense of distributions. We do so initially by proving a fact about integrable functions.

\begin{theorem}\label{bvdistl1}
Every $f \in L^1(D)$ with nontangential boundary value $f_+\in L^1(\p D)$ such that 
\[
 \lim_{r\nearrow 1} \int_0^{2\pi} \left| f(re^{i\theta})  - f_+(e^{i\theta}) \right|  \,d\theta = 0
\]
has a boundary value in the sense of distributions $f_b$ and $f_+=f_b$ as distributions.  
\end{theorem}

\begin{proof}
    For $\varphi \in C^\infty(\p D)$ and $0 < r < 1$, observe that 
    \begin{align*}
        \left|\int_0^{2\pi} f(re^{i\theta}) \varphi(\theta) \,d\theta \right|
        &\leq \sup_\theta|\varphi| \int_0^{2\pi} |f(re^{i\theta})|\,d\theta.
    \end{align*}
    So, 
    \begin{align*}
   \left |\langle f_b, \varphi \rangle \right| :&= \left|\lim_{r \nearrow 1} \int_0^{2\pi} f(re^{i\theta}) \varphi(\theta) \,d\theta \right|\\
    &\leq \sup_\theta|\varphi| \lim_{r \nearrow 1}\int_0^{2\pi} |f(re^{i\theta})|\,d\theta  \\
    &=  \sup_\theta|\varphi|\int_0^{2\pi} |f_+(e^{i\theta})|\,d\theta< \infty,
    \end{align*}
since $f$ converges to $f_+$ in the $L^1$ norm. Also, we have 
 \begin{align*}
    \left |\langle f_b  - f_+, \varphi \rangle \right| &\leq
    \lim_{r\nearrow 1} \int_0^{2\pi} \left| f(re^{i\theta})  - f_+(e^{i\theta}) \right|  \,d\theta = 0.
 \end{align*}
\end{proof}

\begin{corr}
For $n$ a positive integer and $A \in W^{n-1,\infty}(D)$, every $f \in H^{n,1}_A(D)$ has a boundary value in the sense of distributions $f_b$, and $f_b = f_+$ as distributions, where $f_+$ is the nontangential boundary value of $f$.
\end{corr}

Now, we restrict our view to $A \in C^\infty(\overline{D})$. We know that if $A \in C^\infty(\overline{D})$, then $A \in W^{n-1,\infty}(D)$ and from \cite{Vek}, $
\psi(z) = -\frac{1}{\pi} \iint_D \frac{A(\zeta)}{\zeta - z}\,d\xi\,d\eta
\in C^\infty(\overline{D})$. This will lead to a quick proof of the next result. 

\begin{theorem}
For $n$ a positive integer and $A \in C^\infty(\overline{D})$, every solution $f$ of 
\[
\left(\frac{\p }{\p\z}- A\right)^n  f= 0
\]
with representation
\[
f(z) = e^{\psi(z)}\sum_{k=0}^{n-1} \z^k f_k(z),
\]
where
\[
\psi(z) = -\frac{1}{\pi} \iint_D \frac{A(\zeta)}{\zeta - z}\,d\xi\,d\eta
\]
and $f_k \in H_b$, for every $k$, has a boundary value in the sense of distributions.

\end{theorem}

\begin{proof}
Observe that, for $0 < r < 1$ and $\varphi \in C^\infty(\p D)$, we have
\begin{align*}
\int_0^{2\pi} f(re^{i\theta}) \varphi(\theta) \,d\theta
&= \int_0^{2\pi} e^{\psi(re^{i\theta})}\sum_{k=0}^{n-1} r^k e^{-ik\theta} f_k(re^{i\theta}) \varphi(\theta) \,d\theta \\
&=  \sum_{k=0}^{n-1} r^k \int_0^{2\pi} e^{\psi(re^{i\theta})-ik\theta}  f_k(re^{i\theta}) \varphi(\theta) \,d\theta.
\end{align*}
Since each $f_k \in H_b$ and $e^\psi$ is $C^\infty(\p D)$ for every $r$, $0\leq r \leq 1$, it follows that 
\[
\lim_{r\nearrow 1} \left|\int_0^{2\pi} f_k(re^{i\theta}) e^{\psi(re^{i\theta})-ik\theta}\varphi(\theta) \,d\theta \right|< \infty,
\]
for each $k$. Thus, 
\[
\lim_{r \nearrow 1} \left|\int_0^{2\pi} f(re^{i\theta}) \varphi(\theta) \,d\theta \right| \leq \lim_{r \nearrow 1}  \sum_{k=0}^{n-1} \left|\int_0^{2\pi} e^{\psi(re^{i\theta})-ik\theta}  f_k(re^{i\theta}) \varphi(\theta) \,d\theta \right| < \infty.
\]

\end{proof}

\begin{corr}
For $n$ a positive integer, $0< p \leq 1$, and $A \in C^{\infty}(\overline{D})$, every $f \in \Mpn$ has a boundary value in the sense of distributions.
\end{corr}

\section{Schwarz boundary value problem}\label{sbvp}

In this section, we exploit a construction from \cite{WB2} to solve a higher-order Schwarz boundary value problem for meta-analytic functions. From \cite{WB2}, we have the following. 

\begin{theorem}[Theorem 2.20 \cite{WB2}]\label{schwarznoint}
For $n$ a positive integer, the Schwarz problem 
    \begin{numcases}{}
        \frac{\p^n f_n}{\p\z^n} = 0 \nonumber\\
        \re\left\{ \left( \frac{\p^k f_n}{\p\z^k}\right)_b \right\} = \re\left\{ (f_{n-k})_b\right\} =  \re\left\{ (h_{n-1-k})_b - \sum_{\ell = 1}^{n-k-1} \frac{(-1)^\ell}{\ell!} e^{-i\ell(\cdot)} (f_{n-\ell})_b\right\} \\
        \im\left\{ \frac{\p^k f_n}{\p\z^k}(0) \right\} = c_{n-1-k} \nonumber
    \end{numcases}
where each $f_k$ solves the Schwarz problem
    \begin{numcases}{}
        \frac{\p f_k}{\p\z} = f_{k-1} \nonumber\\
        \re\left\{ (f_k)_b\right\} = \re\left\{ (h_{k-1})_b - \sum_{\ell = 1}^{k-1} \frac{(-1)^\ell}{\ell!} e^{-i\ell(\cdot)} (f_{k-\ell})_b \right\} \\
        \im\left\{ f_k(0)\right\} =  c_{k-1} \nonumber
    \end{numcases}
with $h_{k-1} \in H_b$, and $c_{k-1} \in \R$,
for each $k$ with $1 \leq k \leq n$, is solved by 
\begin{align*}
 f_n(z) = ic_{n-1} - I_{n-1} + \frac{1}{2\pi} \langle (h_{n-1})_b, P_r(\theta - \cdot) \rangle - \sum_{\ell = 1}^{n-1} \frac{(-1)^\ell}{\ell!} \z^\ell f_{n-\ell}(z), 
\end{align*}
where each $f_k$ is described by 
\[
f_k(z) = ic_{k-1} - I_{k-1} + \frac{1}{2\pi} \langle (h_{k-1})_b, P_r(\theta - \cdot) \rangle - \sum_{\ell = 1}^{k-1} \frac{(-1)^\ell}{\ell!} \z^\ell f_{k-\ell}(z),
\]
with
\[
I_{k-1} := \frac{i}{2\pi} \langle \im\{(h_{k-1})_b\}, 1\rangle,
\] 
for $1 \leq k \leq n$, and $f_0 \equiv 0$.
\end{theorem}

Theorem \ref{schwarznoint} extends Theorem 3.4 in \cite{Beg} where the boundary condition is required to be in terms of a continuous function.  While the boundary condition is not as straightforward as the one in Theorem 3.4 of \cite{Beg}, the boundary condition is the direct result of the natural extension of the representation formula found in Theorem 3.4 of \cite{Beg} when utilizing the Poisson integral representation of boundary values in the sense of distributions of holomorphic functions from Theorem 3.1 of \cite{GHJH2}. We appeal to the constructed sequence of functions $\{f_k\}$ in Theorem \ref{schwarznoint} to extend this theorem to the meta-analytic setting. The following theorem extends the results of Theorem 3.2 and Theorem 3.3 in \cite{metahardy}, in the case where $a \equiv 1$ and $b \equiv 0$ in those theorems, to solve a Schwarz boundary value problem for meta-analytic functions with boundary condition only in terms of boundary values in the sense of distributions.

\begin{theorem}\label{schwarzmeta}
For $n$ a positive integer and $A \in W^{n-1,q}(D)$, $q>2$, the Schwarz problem 
    \begin{numcases}{}
        \left(\frac{\p }{\p\z} - A\right)^n w = 0 \label{pde}\\
        \re\left\{ \left( \frac{\p^k w}{\p\z^k}/ e^{\psi}\right)_b \right\} = \re\left\{ (f_{n-k})_b\right\} =  \re\left\{ (h_{n-1-k})_b - \sum_{\ell = 1}^{n-k-1} \frac{(-1)^\ell}{\ell!} e^{-i\ell(\cdot)} (f_{n-\ell})_b\right\} \label{bc}\\
        \im\left\{ \left(\frac{\p^k w}{\p\z^k}/e^{\psi}\right)(0) \right\} = c_{n-1-k} \label{impt}
    \end{numcases}
where each $f_k$ solves the Schwarz problem
    \begin{numcases}{}
        \frac{\p f_k}{\p\z} = f_{k-1} \nonumber\\
        \re\left\{ (f_k)_b\right\} = \re\left\{ (h_{k-1})_b - \sum_{\ell = 1}^{k-1} \frac{(-1)^\ell}{\ell!} e^{-i\ell(\cdot)} (f_{k-\ell})_b \right\} \nonumber\\
        \im\left\{ f_k(0)\right\} =  c_{k-1} \nonumber
    \end{numcases}
with $h_{k-1} \in H_b$, and $c_{k-1} \in \R$,
for each $k$ with $1 \leq k \leq n$, is solved by 
\begin{align*}
w(z) = e^{\psi(z)}\left[ic_{n-1} - I_{n-1} + \frac{1}{2\pi} \langle (h_{n-1})_b, P_r(\theta - \cdot) \rangle - \sum_{\ell = 1}^{n-1} \frac{(-1)^\ell}{\ell!} \z^\ell f_{n-\ell}(z)\right], 
\end{align*}
where each $f_k$ is described by 
\[
f_k(z) = ic_{k-1} - I_{k-1} + \frac{1}{2\pi} \langle (h_{k-1})_b, P_r(\theta - \cdot) \rangle - \sum_{\ell = 1}^{k-1} \frac{(-1)^\ell}{\ell!} \z^\ell f_{k-\ell}(z),
\]
with
\[
I_{k-1} := \frac{i}{2\pi} \langle \im\{(h_{k-1})_b\}, 1\rangle,
\] 
for $1 \leq k \leq n$, and $f_0 \equiv 0$
and 
\[
\psi(z) = -\frac{1}{\pi} \iint_D \frac{A(\zeta)}{\zeta - z}\,d\xi\,d\eta.
\]
\end{theorem}

\begin{proof}
The existence of such a sequence $\{f_k\}$ is proved in \cite{WB2}. By direct computation, $w$ is a solution of the equation (\ref{pde}). Conditions (\ref{bc}) and (\ref{impt}) are satisfied by $w/e^\psi$ by observing that $w/e^\psi$ is a solution of the Schwarz problem  solved in Theorem \ref{schwarznoint}. 
\end{proof}

\begin{comm}From \cite{WB2}, if the $h_k$ in the hypothesis of Theorem \ref{schwarznoint} are such that $h_k \in H^{p_k}(D)$, then the solution $w$ in Theorem \ref{schwarznoint} is an element of $H^{n,p}(D)$, where $p:= \min_k\{p_k\}$. Consequently, if the $h_k$ in the hypothesis of Theorem \ref{schwarzmeta} are chosen so that $h_k \in H^{p_k}(D)$ and $A \in W^{n-1,\infty}(D)$, then the solution $w$ in Theorem \ref{schwarzmeta} is an element of $H^{n,p}_A(D)$, where $p:= \min_{k}\{p_k\}$.  
\end{comm}

In the case of $A \in C^\infty(\overline{D})$, we solve a Schwarz boundary value problem where the solution is meta-analytic and the boundary condition and the pointwise condition at $z = 0$ can be in terms of the solution, not its poly-analytic factor. 

\begin{theorem}\label{schwarzmetanodivide}
For $n$ a positive integer and $A \in C^\infty(\overline{D})$, the Schwarz problem 
    \begin{numcases}{}
        \left(\frac{\p }{\p\z} - A\right)^n w = 0, \label{pde1}\\
        \re\left\{ \left( \left(\frac{\p }{\p\z}-A\right)^k w \right)_b \right\} \nonumber\\
        =  \re\left\{ e^{\psi(e^{i(\cdot)})}\left(ic_{n-1-k} - I_{n-1-k}+(h_{n-1-k})_b - \sum_{\ell = 1}^{n-k-1} \frac{(-1)^\ell}{\ell!} e^{-i\ell(\cdot)} (f_{n-\ell})_b\right)\right\}, \label{bc1}\\
        \im\left\{ \left(\left(\frac{\p }{\p\z}-A\right)^k w\right)(0) \right\} = e^{\psi(0)}c_{n-1-k} \label{impt1}
    \end{numcases}
where each $f_k$ solves the Schwarz problem
    \begin{numcases}{}
        \frac{\p f_k}{\p\z} = f_{k-1} \nonumber\\
        \re\left\{ (f_k)_b\right\} = \re\left\{ (h_{k-1})_b - \sum_{\ell = 1}^{k-1} \frac{(-1)^\ell}{\ell!} e^{-i\ell(\cdot)} (f_{k-\ell})_b \right\} \nonumber\\
        \im\left\{ f_k(0)\right\} =  c_{k-1} \nonumber
    \end{numcases}
with $h_{k-1} \in H_b$, and $c_{k-1} \in \R$,
for each $k$ with $1 \leq k \leq n$, is solved by 
\begin{align*}
w(z) = e^{\psi(z)}\left[ic_{n-1} - I_{n-1} + \frac{1}{2\pi} \langle (h_{n-1})_b, P_r(\theta - \cdot) \rangle - \sum_{\ell = 1}^{n-1} \frac{(-1)^\ell}{\ell!} \z^\ell f_{n-\ell}(z)\right], 
\end{align*}
where each $f_k$ is described by 
\[
f_k(z) = ic_{k-1} - I_{k-1} + \frac{1}{2\pi} \langle (h_{k-1})_b, P_r(\theta - \cdot) \rangle - \sum_{\ell = 1}^{k-1} \frac{(-1)^\ell}{\ell!} \z^\ell f_{k-\ell}(z),
\]
with
\[
I_{k-1} := \frac{i}{2\pi} \langle \im\{(h_{k-1})_b\}, 1\rangle,
\] 
for $1 \leq k \leq n$, and $f_0 \equiv 0$
and 
\[
\psi(z) = -\frac{1}{\pi} \iint_D \left( \frac{A(\zeta)}{\zeta}\frac{\zeta + z}{\zeta - z} + \frac{\overline{A(\zeta)}}{\overline{\zeta} }\frac{1+z\overline{\zeta}}{1-z\overline{\zeta}} \right)\,d\xi\,d\eta.
\]
\end{theorem}

\begin{proof}
    The constructed $w$ solves (\ref{pde1}) by the same direct computation as in the proof of Theorem \ref{schwarzmeta}. Since $A\in C^\infty(\overline{D})$ implies $e^\psi \in C^\infty(\overline{D})$, it follows, by Theorem \ref{bvdistl1}, that $(e^\psi)_b = e^\psi|_{\p D}$. Since 
    \[
        \left(\frac{\p }{\p\z}-A\right)^k w  = e^{\psi(z)}\left[ic_{n-1-k} - I_{n-1-k} + \frac{1}{2\pi} \langle (h_{n-1-k})_b, P_r(\theta - \cdot) \rangle - \sum_{\ell = 1}^{n-1-k} \frac{(-1)^\ell}{\ell!} \z^\ell f_{n-\ell}(z)\right],
    \]
    it follows that 
    \begin{align*}
        &  \left(\left(\frac{\p }{\p\z}-A\right)^k w\right)_b\\
        &= \left( e^{\psi(z)}\left[ic_{n-1-k} - I_{n-1-k} + \frac{1}{2\pi} \langle (h_{n-1-k})_b, P_r(\theta - \cdot) \rangle - \sum_{\ell = 1}^{n-1-k} \frac{(-1)^\ell}{\ell!} \z^\ell f_{n-\ell}(z)\right]\right)_b\\
        &= e^{\psi(e^{i(\cdot)})}\left(ic_{n-1-k} - I_{n-1-k}+(h_{n-1-k})_b - \sum_{\ell = 1}^{n-k-1} \frac{(-1)^\ell}{\ell!} e^{-i\ell(\cdot)} (f_{n-\ell})_b\right)
    \end{align*}
    and 
    \begin{align*}
        & \re\left\{ \left( \left(\frac{\p }{\p\z}-A\right)^k w \right)_b \right\} \\
        &=  \re\left\{ e^{\psi(e^{i(\cdot)})}\left(ic_{n-1-k} - I_{n-1-k}+(h_{n-1-k})_b - \sum_{\ell = 1}^{n-k-1} \frac{(-1)^\ell}{\ell!} e^{-i\ell(\cdot)} (f_{n-\ell})_b\right)\right\}.
    \end{align*}
    Hence, (\ref{bc1}) is satisfied. Now, since $e^{\psi(0)}$ is real, it follows that
    \begin{align*}
        &\left(\left(\frac{\p }{\p\z}-A\right)^k w\right) (0) \\
        &= e^{\psi(0)}\left[ic_{n-1-k} - I_{n-1-k} + \frac{1}{2\pi} \langle (h_{n-1-k})_b, 1 \rangle - \sum_{\ell = 1}^{n-1-k} \frac{(-1)^\ell}{\ell!} 0^\ell f_{n-\ell}(0)\right] \\
        &= e^{\psi(0)}\left[ic_{n-1-k} + \re\left\{\frac{1}{2\pi} \langle (h_{n-1-k})_b, 1 \rangle\right\}\right]\\
        &= e^{\psi(0)}\re\left\{\frac{1}{2\pi} \langle (h_{n-1-k})_b, 1 \rangle\right\}+ ie^{\psi(0)}c_{n-1-k}\\
    \end{align*}
    and 
    \begin{align*}
        \im\left\{\left(\left(\frac{\p }{\p\z}-A\right)^k w\right) (0)\right\} 
        &= e^{\psi(0)}c_{n-1-k}.
    \end{align*}
    Therefore, (\ref{impt1}) is satisfied.
\end{proof}

\printbibliography

@article{GHJH2,
    author = "Hoepfner, G. and  Hounie, J.",
    title = "Atomic Decompositions of Holomorphic Hardy Spaces in $\mathbb{S}^1$ and Applications",
    journal = "Lecture Notes of Seminario Interdisciplinare di Matematica 7",
    pages = "189-206",
    year = "2008",
}

@book {Vek,
    AUTHOR = {Vekua, I. },
     TITLE = {Generalized analytic functions},
 PUBLISHER = {Pergamon Press, London-Paris-Frankfurt; Addison-Wesley
              Publishing Co., Inc., Reading, Mass.},
      YEAR = {1962},
     PAGES = {xxix+668},
   MRCLASS = {30.81},
  MRNUMBER = {0150320},
MRREVIEWER = {A. E. Heins},
}

@book {KlimBook,
    AUTHOR = {Klimentov, S. },
     TITLE = {{\cyr Granichnye svo\u{i}} {\cyr stva obobshchennykh analiticheskikh funktsi\u{i}}},
    SERIES = {{\cyr Itogi Nauki. Yug Rossii. Matematicheskaya Monografiya} [Progress in
              Science. South Russia. Mathematical Monograph]},
    VOLUME = {7},
 PUBLISHER = {Yuzhny\u{\i} Matematicheski\u{\i} Institut, \\Vladikavkazski\u{\i} Nauchny\u{\i}
              Tsentr, Rossi\u{\i}skaya Akademiya Nauk i RSO-A, Vladikavkaz},
      YEAR = {2014},
     PAGES = {200},
      ISBN = {978-5-904695-26-2},
   MRCLASS = {30-02 (30E25 30G20 30H10 30H15 30H35)},
  MRNUMBER = {3408907},
MRREVIEWER = {Kuzman Adzievski},
}

@article {polyhardy,
    AUTHOR = {Wang, Y.},
     TITLE = {On {H}ilbert-type boundary-value problem of poly-{H}ardy class
              on the unit disc},
   JOURNAL = {Complex Var. Elliptic Equ.},
  FJOURNAL = {Complex Variables and Elliptic Equations. An International
              Journal},
    VOLUME = {58},
      YEAR = {2013},
    NUMBER = {4},
     PAGES = {497--509},
      ISSN = {1747-6933},
   MRCLASS = {30E25 (30H10 31A25 45E05)},
  MRNUMBER = {3038743},
MRREVIEWER = {Kuzman Adzievski},
       DOI = {10.1080/17476933.2011.636809},
       URL = {https://doi.org/10.1080/17476933.2011.636809},
}

@book {Balk,
    AUTHOR = {Balk, M.},
     TITLE = {Polyanalytic functions},
    SERIES = {Mathematical Research},
    VOLUME = {63},
 PUBLISHER = {Akademie-Verlag, Berlin},
      YEAR = {1991},
     PAGES = {197},
      ISBN = {3-05-501292-5},
   MRCLASS = {30G20 (31A30)},
  MRNUMBER = {1184141},
MRREVIEWER = {Heinrich Begehr},
}

@article {polyhardyrep,
    AUTHOR = {Du, Z. and Guo, G. and Wang, N.},
     TITLE = {Decompositions of functions and {D}irichlet problems in the
              unit disc},
   JOURNAL = {J. Math. Anal. Appl.},
  FJOURNAL = {Journal of Mathematical Analysis and Applications},
    VOLUME = {362},
      YEAR = {2010},
    NUMBER = {1},
     PAGES = {1--16},
      ISSN = {0022-247X},
   MRCLASS = {31A30 (31A25 31A35)},
  MRNUMBER = {2557663},
MRREVIEWER = {Valery V. Karachik},
       DOI = {10.1016/j.jmaa.2009.07.048},
       URL = {https://doi.org/10.1016/j.jmaa.2009.07.048},
}

@article {metahardy,
    AUTHOR = {Ku, M. and He, F. and Wang, Y.},
     TITLE = {Riemann-{H}ilbert problems for {H}ardy space of meta-analytic
              functions on the unit disc},
   JOURNAL = {Complex Anal. Oper. Theory},
  FJOURNAL = {Complex Analysis and Operator Theory},
    VOLUME = {12},
      YEAR = {2018},
    NUMBER = {2},
     PAGES = {457--474},
      ISSN = {1661-8254},
   MRCLASS = {30E25 (30H10)},
  MRNUMBER = {3756167},
MRREVIEWER = {Kuzman Adzievski},
       DOI = {10.1007/s11785-017-0705-1},
       URL = {https://doi.org/10.1007/s11785-017-0705-1},
}

@book {Duren,
    AUTHOR = {Duren, P.},
     TITLE = {Theory of {$H^{p}$} spaces},
    SERIES = {Pure and Applied Mathematics, Vol. 38},
 PUBLISHER = {Academic Press, New York-London},
      YEAR = {1970},
     PAGES = {xii+258},
   MRCLASS = {46.30 (30.00)},
  MRNUMBER = {0268655},
MRREVIEWER = {D. Sarason},
}

@book {BegBook,
    AUTHOR = {Begehr, H.},
     TITLE = {Complex analytic methods for partial differential equations},
      NOTE = {An introductory text},
 PUBLISHER = {World Scientific Publishing Co., Inc., River Edge, NJ},
      YEAR = {1994},
     PAGES = {x+273},
      ISBN = {981-02-1550-9},
   MRCLASS = {35C15 (30E20 35A20)},
  MRNUMBER = {1314196},
MRREVIEWER = {Guo Chun Wen},
       DOI = {10.1142/2162},
       URL = {https://doi.org/10.1142/2162},
}

@article {Beg,
    AUTHOR = {Begehr, H. and Schmersau, D.},
     TITLE = {The {S}chwarz problem for polyanalytic functions},
   JOURNAL = {Z. Anal. Anwendungen},
  FJOURNAL = {Zeitschrift f\"{u}r Analysis und ihre Anwendungen. Journal for
              Analysis and its Applications},
    VOLUME = {24},
      YEAR = {2005},
    NUMBER = {2},
     PAGES = {341--351},
      ISSN = {0232-2064},
   MRCLASS = {30E25 (30G20 31A30 35J40)},
  MRNUMBER = {2174027},
MRREVIEWER = {Wolfgang Tutschke},
       DOI = {10.4171/ZAA/1244},
       URL = {https://doi.org/10.4171/ZAA/1244},
}

@unpublished{ WB,
AUTHOR = {Blair, W.},
TITLE = {An atomic representation for Hardy classes of solutions to nonhomogeneous Cauchy-Riemann equations},
NOTE = {To appear, \textit{J Geom Anal}},
%YEAR = 2023

}

@unpublished{ WB2,
AUTHOR = {Blair, W.},
TITLE = {The Schwarz boundary value problem for boundary values in the sense of distributions},
NOTE = {Submitted},
YEAR = 2023

}
\end{document}